\documentclass[preprint]{elsarticle}
\usepackage{amsfonts, amstext, amsmath, amssymb, mathrsfs, amsthm}

\begin{document}

\begin{frontmatter}

\title{The Number of Harmonic Frames of Prime Order\tnoteref{t1}}
\tnotetext[t1]{Submitted to {\it Linear Algebra and its Applications}
  15 April 2009. Accepted 22 September 2009. Published in Volume 432,
  Issue 5, 15 February 2010, Pages 1105-1125. LaTeX file last compiled
  2 September 2012.}

\author{Matthew J. Hirn}
\ead{matthew.hirn@yale.edu}
\ead[url]{www.math.yale.edu/$\sim$mh644}

\address{
Yale University \\
Department of Mathematics \\
P.O. Box 208283 \\
New Haven, Connecticut 06520-8283 \\
USA
}

\begin{abstract}
Harmonic frames of prime order are investigated. The primary focus is the enumeration of inequivalent harmonic frames, with the exact number given by a recursive formula. The key to this result is a one-to-one correspondence developed between inequivalent harmonic frames and the orbits of a particular set.  Secondarily, the symmetry group of prime order harmonic frames is shown to contain a subgroup consisting of a diagonal matrix as well as a permutation matrix, each of which is dependent on the particular harmonic frame in question.
\end{abstract}

\begin{keyword}
finite unit norm tight frame, harmonic frame, symmetry group of frame

{\it AMS classification codes:}  primary 42C15; secondary 05A05, 20E07
\end{keyword}

\end{frontmatter}

\numberwithin{equation}{section}
\parindent 0pt

\bibliographystyle{plain}

\theoremstyle{plain}
\newtheorem{theorem}{Theorem}[section]
\newtheorem{lemma}[theorem]{Lemma}
\newtheorem{conjecture}[theorem]{Conjecture}
\newtheorem{proposition}[theorem]{Proposition}
\newtheorem{corollary}[theorem]{Corollary}

\theoremstyle{definition}
\newtheorem{definition}[theorem]{Definition}
\newtheorem{example}[theorem]{Example}
\newtheorem{remark}[theorem]{Remark}

\section{Introduction}\label{sec: intro}

\subsection{DFT-FUNTFs}

Let $N,d \in \mathbb{N} = \{1,2,3,\ldots\}$, $d \leq N$, and consider
the space $\mathbb{C}^d$. An ordered set $X = (x_1,\ldots,x_N) \subset \mathbb{C}^d$ is a {\it finite frame} for $\mathbb{C}^d$ if there exist constants $A,B > 0$ such that
\begin{equation}\label{eqn:frame}
A\lVert f \rVert^2 \leq \sum_{j=1}^N |\langle f,x_j\rangle |^2 \leq B\lVert f\rVert^2, \quad \forall \enspace f \in \mathbb{C}^d.
\end{equation}
The numbers $A,B$ are called the {\it frame bounds}. It is well known that any spanning set is a frame for $\mathbb{C}^d$, while every frame is itself a spanning set. A frame is {\it tight} if one can choose $A = B$ in the definition, i.e., if
\begin{equation}\label{eqn:tight}
\sum_{j=1}^N |\langle f,x_j\rangle |^2 = A\lVert f\rVert^2, \quad \forall \enspace f \in \mathbb{C}^d.
\end{equation}
Finally, a frame is {\it unit norm} if
\begin{equation}\label{eqn:unit norm}
\lVert x_j \rVert = 1, \quad \forall \enspace j=1,\ldots ,s.
\end{equation}
If $X$ satisfies (\ref{eqn:tight}) and (\ref{eqn:unit norm}), then we say $X$ is a {\it finite unit norm tight frame} (FUNTF) for $\mathbb{C}^d$. In this case, the frame bounds satisfy $A = B = N/d$. In particular, if $X$ is a FUNTF with frame bounds $A = B = 1$, then $X$ is an orthonormal basis.\\

There has been much literature on the subject of finite tight frames,
see for example \cite{benedetto:fntf, kovacevic:entfwe,
  kovacevic:qfewe, waldron:itf, zimmermann:ntffd} and references therein. One special class of FUNTFs is obtained by considering the un-normalized Discrete Fourier Transform (DFT) matrix,
\begin{equation}\label{eqn:dft matrix}
D_N := (e^{2\pi imn/N})_{m,n = 0}^{N-1}.
\end{equation}
Choose any distinct $d$ columns, $n_1,\ldots,n_d$, of $D_N$, with $n_j
\in \{0,\ldots,N-1\}$ for each $j = 1,\ldots,d$, and form the following $N$ vectors in $\mathbb{C}^d$,
\begin{equation}\label{eqn:dft-funtf}
\phi_m = \frac{1}{\sqrt{d}} (e^{2\pi imn_1/N}, e^{2\pi imn_2/N}, \ldots, e^{2\pi imn_d/N}) \in \mathbb{C}^d, \quad m = 0,1,\ldots,N-1.
\end{equation}
\\
It is well known that $\Phi = (\phi_0,\ldots ,\phi_{N-1})$ is a FUNTF for $\mathbb{C}^d$.  Any frame of this type is called a DFT-FUNTF, and collectively, they form a subset of a special class of frames known as harmonic frames (see \cite{waldron:tfts, waldron:ocahfonvicd} as well as section \ref{subsec:harmonic frames}). We call the column choices $n_1,\ldots,n_d$ the {\it generators} of the frame.\\

A basic way of counting the number of DFT-FUNTFs is inspired by the following observation. For any vector $f \in \mathbb{C}^d$, the frame $\Phi$ gives the following representation of $f$:
\begin{equation*}
f \mapsto \left(\langle f,\phi_m \rangle\right)_{m=0}^{N-1} \in \mathbb{C}^N.
\end{equation*}
Therefore, even a re-indexing of the frame would change the
representation it gives for a fixed $f$. Thus, we could count the
number of ordered DFT-FUNTFs. To accomplish this task, we observe that
there are $N$ columns in $D_N$ and we select $d$ of them.  Since each ordered
combination of column choices $n_1,\ldots ,n_d$ gives a distinct
frame, there are $N(N-1)\cdots (N-d+1)$ ordered DFT-FUNTFs.\\

There are of course other ways by which we may distinguish frames, and we shall consider two others here. The first is a natural counterpart to the ordered counting scheme, namely, counting the number of DFT-FUNTFs considered as unordered sets of vectors. The techniques developed for this method will then be expanded to our main goal, which is to count all inequivalent harmonic frames of prime order, where two harmonic frames shall be considered equivalent if one is the unitary transformation of the other. As we shall see, this amounts to counting the number of inequivalent DFT-FUNTFs.\\

There has been some interest in harmonic frames in the literature, see   \cite{bolcskei:guf,waldron:tfts}. In particular, \cite{waldron:ocahfonvicd} presents a computer program for generating all equivalence classes of harmonic frames for a given $N$ and $d$, where there is a limit on the size of either due to computational considerations. From this program, the authors conjecture that there are $\mathcal{O}(N^{d-1})$ inequivalent harmonic frames. The content of this paper is to not only validate this conjecture for the case when $N$ is a prime number, but in fact give an exact formula for the number of harmonic frames in this case. Furthermore, we examine the structure of prime order harmonic frames via their symmetry group.\\

An outline of this paper is as follows: the remainder of section \ref{sec: intro} reviews some algebraic theory and examines the problem of counting unordered DFT-FUNTFs. Section \ref{sec:the number of harmonic frames of prime order} defines harmonic frames and presents the main result of this paper. In section \ref{sec:harmonic frames and orbits} we define an equivalence relation that is equivalent to (\ref{eqn:equiv2}) and then use this to develop a correspondence between inequivalent harmonic frames and the orbits of a particular set. Section \ref{sec:the number of orbits of A} counts the number of orbits of this particular set, thus giving a formula for the number of inequivalent harmonic frames. The structure of the symmetry group is handled in section \ref{sec:the symmetry group}, and section \ref{sec:closing remarks} contains a few concluding remarks. \\

\subsection{Algebra Review}

Denote the additive group of integers mod $N$ by $\mathbb{Z}_N$, and set
\begin{equation*}
\mathbb{Z}_N^d := \underbrace{\mathbb{Z}_N \times \cdots \times \mathbb{Z}_N}_{d \text{ times}}.
\end{equation*}
Furthermore, let $\mathbb{Z}_N^{\times}$ denote the group of units of $\mathbb{Z}_N$, which, when $N$ is prime, is simply the set $\{1,\ldots,N\}$ endowed with multiplication mod $N$. Finally, for $k \in \mathbb{N}$, let $S_k$ denote the group of permutations of $k$ elements. We will also need the following definitions and proposition:

\begin{definition}
A {\it group action} of a group $G$ on a set $S$ is a map $\pi$,
\begin{eqnarray*}
\pi :G\times S & \rightarrow & S\\
(g,s) & \mapsto & g \cdot s,
\end{eqnarray*}
satisfying the following properties:
\begin{enumerate}
\renewcommand{\labelenumi}{\arabic{enumi})}
\item
$g_1 \cdot (g_2 \cdot s) = (g_1g_2) \cdot s \enspace \forall \enspace g_1,g_2 \in G, \enspace s \in S$,
\item
$1 \cdot s = s \enspace \forall \enspace s \in S$.
\end{enumerate}
\end{definition}
\begin{definition}
Let $S$ be some set and let $G$ be a group. Furthermore, let $\pi :G
\times S \rightarrow S$ be a group action. For each $s \in S$ the
{\it stabilizer} of $s$ in $G$ is the subgroup of $G$ that fixes the
element $s$:
\begin{equation*}
G_s := \{g \in G : g \cdot s = s\}.
\end{equation*}
\end{definition}
\begin{proposition}\label{prop:orbit}
Let $G$ be a group acting on the nonempty set $S$. The relation on $S$
defined by:
\begin{equation*}
s_1 \sim s_2 \iff  s_1 = g \cdot s_2 \text{ for some } g \in G
\end{equation*}
is an equivalence relation. For each $s \in S$, the number of elements
in the equivalence class containing $s$ is $|G:G_s|$, the index of the
stabilizer of $s$.
\end{proposition}
Note, when $G$ is a finite group,
\begin{equation*}
|G:G_s| = \frac{|G|}{|G_s|}.
\end{equation*}
\begin{definition}
Let $G$ be a group acting on the nonempty set $S$. The equivalence class $\mathcal{O}_s := \{g\cdot s : g \in G\}$ is called the {\it orbit} of $G$ containing $s$.
\end{definition}
As such, the orbits of a group action partition the set $S$. We are now ready to count the number of prime order DFT-FUNTFs, considered as unordered sets.  The basic structure of the argument in subsection \ref{subsec:ordered} will be used when we count all harmonic frames of prime order, albeit with added complexity.

\subsection{The Number of Unordered DFT-FUNTFs}\label{subsec:ordered}

It is often the case that we would like to consider a frame as a set, where the order of elements does not matter.  Given two ordered DFT-FUNTFs $\Phi = (\phi_0,\ldots ,\phi_{N-1})$ and $\Psi = (\psi_0,\ldots ,\psi_{N-1})$, we define the following equivalence relation:
\begin{equation}\label{eqn:equiv1}
\Phi \sim_1 \Psi \iff \exists \thinspace \sigma \in S_N \enspace
s.t. \enspace \phi_m = \psi_{\sigma (m)}, \quad \forall \enspace m = 0,\ldots,N-1.
\end{equation}
(\ref{eqn:equiv1}) merely formalizes our consideration of frames as sets. An equivalence class of (\ref{eqn:equiv1}) will be denoted in the usual way, that is $\Phi = \{\phi_0,\ldots,\phi_{N-1}\}$. In this subsection, we count the number of DFT-FUNTFs of prime order under (\ref{eqn:equiv1}). First, however, we must change our perspective on the problem.
\begin{remark}
For the rest of the paper we will only consider unordered DFT-FUNTFs, and as such from now on $\Phi$ will denote $\{\phi_0,\ldots,\phi_{N-1}\}$.
\end{remark}

\subsubsection{DFT-FUNTFs and Orbits}\label{subsec:dftfuntfs and orbits}

First notice that every DFT-FUNTF contains the vector $\phi_0 =
\frac{1}{\sqrt{d}}(1,\ldots,1)\in \mathbb{C}^d$, and so when comparing
two such frames we need not consider this vector. Thus we will
only compare sets of the form
\begin{equation*}
\Phi' = \Phi - \{\phi_0\}.
\end{equation*}
Define the set $\tilde{\mathbb{Z}}_N^d$ as
\begin{equation*}
\tilde{\mathbb{Z}}_N^d := \{n = (n_1,\ldots, n_d) \in \mathbb{Z}_N^d : n_i \neq n_j, \enspace \forall \enspace i \neq j\}.
\end{equation*}
There is a one-to-one correspondence between the vectors $\phi_m$, $m \neq 0$, and the elements of $\tilde{\mathbb{Z}}_N^d$. Considering $\mathbb{Z}_N^{\times}$ as a group and $\tilde{\mathbb{Z}}_N^d$ as a set, we define the group action $\pi_1$ as:
\begin{eqnarray*}
\pi_1 :\mathbb{Z}_N^{\times} \times \tilde{\mathbb{Z}}_N^d & \rightarrow & \tilde{\mathbb{Z}}_N^d\\
(m,n) & \mapsto &  m \cdot n := (mn_1,\ldots,mn_d).
\end{eqnarray*}
The orbits of $\pi_1$ are then the sets
\begin{equation*}
\mathcal{O}_n = \{m \cdot n = (mn_1,\ldots,mn_d) :  m \in\mathbb{Z}_N^{\times}\}, \quad n \in \tilde{\mathbb{Z}}_N^d.
\end{equation*}
\begin{remark}
For clarity of exposition we shall sometimes use $\Phi_n$ to denote the DFT-FUNTF $\Phi$ and $\phi_{m,n}$ its corresponding elements, where the subscript $n$ emphasizes the generators $n = (n_1,\ldots,n_d)$. 
\end{remark}
The following proposition relates the equivalence classes of (\ref{eqn:equiv1}) and the orbits of $\pi_1$.
\begin{proposition}\label{prop:funtf-orbit1}
There is a one-to-one correspondence between the equivalence classes
of (\ref{eqn:equiv1}) and the orbits of $\pi_1$, i.e. the sets
$\Phi_n$ and $\mathcal{O}_n$ can be identified. We denote this
identification as:
\begin{equation*}
\Phi_n = \{\phi_0,\ldots,\phi_{N-1}\} \longleftrightarrow \mathcal{O}_n.
\end{equation*}
\end{proposition}
\begin{proof}
As noted above, we have:
\begin{equation*}
\Phi \longleftrightarrow \Phi' = \Phi - \{\phi_0\}.
\end{equation*}
Define a function $F$ that maps orbits of $\tilde{\mathbb{Z}}_N^d$ to sets of the form $\Phi'$ as follows:
\begin{equation*}
F(\mathcal{O}_n) = \{\phi_{m,n}\}_{m=1}^N.
\end{equation*}
We must show that $F$ is both one-to-one and onto, however it is clear that $F$ is surjective. Considering then the former, suppose $F(\mathcal{O}_n) = F(\mathcal{O}_{n'})$. This would imply that $\{\phi_{m,n}\}_{m=1}^N = \{\phi_{m',n'}\}_{m'=1}^N$. But then for some $m$ and some $m'$, we would have $(mn_1,\ldots,mn_d) = (m'n'_1,\ldots,m'n'_d)$, i.e. $\mathcal{O}_n \cap \mathcal{O}_{n'} \neq \emptyset$, and so in fact $\mathcal{O}_n = \mathcal{O}_{n'}$.
\end{proof}
\begin{remark}
Given the content of proposition \ref{prop:funtf-orbit1}, we now replace the problem of counting the equivalence classes of (\ref{eqn:equiv1}) with the problem of counting the orbits of $\pi_1$.
\end{remark}

\subsubsection{The Number of Orbits of $\pi_1$}

By proposition \ref{prop:orbit} we see that the orbits of a group action partition the set into disjoint equivalence classes. In particular, the orbits $\mathcal{O}_n$ partition the set $\tilde{\mathbb{Z}}_N^d$. Furthermore, the size of each $\mathcal{O}_n$ is given by $|\mathcal{O}_n| = |\mathbb{Z}_N^{\times}:(\mathbb{Z}_N^{\times})_n|$. Using these facts, we prove the following proposition.
\begin{proposition}\label{thm:pi1 orbits}
Let $N$ be a prime number and $d \leq N$. Then the number of orbits of $\pi_1$ is:
\begin{enumerate}
\renewcommand{\labelenumi}{\arabic{enumi})}
\item
$2$, $\quad \text{if } d = 1$ or $d = N = 2$.
\item
$N(N-2) \cdots (N-d+1)$, $\quad \text{if } d \geq 2$, $N > 2$.
\end{enumerate}
\end{proposition}

\begin{proof}
We first consider the case $d = 1$. For $n = 0$ we have
$(\mathbb{Z}_N^{\times})_0 = \mathbb{Z}_N^{\times}$, and so $|\mathcal{O}_0| = (N -
1)/(N - 1) = 1$. For $n \neq 0$ we see $(\mathbb{Z}_N^{\times})_n =
\{1\}$, and thus $|\mathcal{O}_n| = N - 1$. Since
$|\tilde{\mathbb{Z}}_N^1| = N$, there are only two orbits.\\

Now take $2 \leq d \leq N$. For each $n \in \tilde{\mathbb{Z}}_N^d$ we
have $(\mathbb{Z}_N^{\times})_n = \{1\}$, and thus $|\mathcal{O}_n| =
N - 1$. Therefore the number of orbits is given by $x$, where
\begin{eqnarray*}
|\tilde{\mathbb{Z}}_N^d| & = & x|\mathcal{O}_n|,\\
N(N-1) \cdots (N-d+1) & = & x(N-1).
\end{eqnarray*}
For $N = 2$ and $d = 2$, we see $x = 2$. For $N > 2$ we have $x =
N(N-2) \cdots (N-d+1)$.
\end{proof}
As an addendum to theorem \ref{thm:pi1 orbits}, we note that one of the orbits in the $d=1$ case corresponds to a degenerate DFT-FUNTF. Namely, the orbit $\mathcal{O}_0$ corresponds to the DFT-FUNTF consisting of the single element $\{1\}$.\\

\section{The Number of Harmonic Frames of Prime Order}\label{sec:the number of harmonic frames of prime order}

Using a similar correspondence between harmonic frames and orbits, we count all harmonic frames of prime order up to unitary transformations. We first give some background information.

\subsection{Harmonic Frames}\label{subsec:harmonic frames}

Let $\mathbb{C}^{\times}$ denote the group of units of $\mathbb{C}$, that is the set $\mathbb{C}\backslash\{0\}$ endowed with multiplication.
\begin{definition}
A {\it character} of a group $G$ is a group homomorphism $\xi:G \rightarrow \mathbb{C}^{\times}$ that satisfies
\begin{equation*}
\xi(g_1g_2) = \xi(g_1)\xi(g_2), \quad \forall \enspace g_1,g_2 \in G.
\end{equation*}
\end{definition}
If $G$ is a finite group, then $\xi(g)$ is a $|G|$-th root of unity. A finite abelian group has exactly $|G|$ characters, and considered as vectors in $\mathbb{C}^{|G|}$, it can be shown that they form an orthogonal basis for $\mathbb{C}^{|G|}$. The square matrix with these vectors as rows is referred to as the {\it character table} of $G$. In particular, when $|G| = N$ is prime, then $G \cong \mathbb{Z}_N$, and the character table of $G$ is the un-normalized DFT matrix, $D_N$.
\begin{definition}
Let $G$ be a finite abelian group of order $N$ with characters $(\xi_j)_{j=1}^N$, $J \subseteq \{1,\ldots,N\}$, and $U:\mathbb{C}^{|J|} \rightarrow \mathbb{C}^{|J|}$ unitary. Then the frame for $\mathbb{C}^{|J|}$ given by
\begin{equation*}
\Phi = U\{(\xi_j(g))_{j\in J} : g\in G\}
\end{equation*}
is called a {\it harmonic frame}.
\end{definition}
\begin{remark}\label{rem:harmonicDFT}
When $N$ is prime and $U = I$, the identity matrix, $\Phi$ is a DFT-FUNTF.
\end{remark}
We will also need the following definition and theorem later on.
\begin{definition}
Let $\mathcal{U}(\mathbb{C}^d)$ denote the group of unitary transformations on $\mathbb{C}^d$. The {\it symmetry group} of a FUNTF $\Phi$ for $\mathbb{C}^d$ is the group:
\begin{equation*}
\mathrm{Sym}(\Phi) := \{U \in \mathcal{U}(\mathbb{C}^d) : U\Phi = \Phi\}.
\end{equation*}
For clarity, we emphasize that $U\Phi = \Phi$ is a set equality.
\end{definition}
\begin{theorem}[Vale and Waldron \cite{waldron:ocahfonvicd}]\label{thm:vw05}
A FUNTF $\Phi$ of $N$ vectors for $\mathbb{C}^d$ is harmonic if and only if it is generated by an abelian group $G \subset \mathrm{Sym}(\Phi)$ of order $N$, i.e., $\Phi = G\phi, \enspace \forall \enspace \phi \in \Phi$. 
\end{theorem}

\subsection{The Number of Inequivalent Harmonic Frames}\label{sec: num harmonic frames}
 
Two harmonic frames $\Phi = \{\phi_0,\ldots ,\phi_{N-1}\} \subset \mathbb{C}^d$ and $\Psi = \{\psi_0,\ldots ,\psi_{N-1}\} \subset \mathbb{C}^d$ are said to be {\it equivalent} if the following equivalence relation holds:
\begin{equation}\label{eqn:equiv2}
\Phi \sim_2 \Psi \iff \exists \enspace U\in \mathcal{U}(\mathbb{C}^d) \text{ s.t. } \Phi = U\Psi.
\end{equation}
Once again, we emphasize that the right hand side of (\ref{eqn:equiv2}) is set equality. (\ref{eqn:equiv2}) is a standard form of equivalence in much of the literature when dealing with frames. Recently, \cite{waldron:ocahfonvicd} conjectured that the number of {\it inequivalent} harmonic frames is $O(N^{d-1})$. We prove this conjecture for $N$ a prime number as a corollary to theorem \ref{thm:main2}, which gives an exact formula for the number of harmonic frames. The proof of theorem \ref{thm:main2} is handled in section \ref{sec:the number of orbits of A}, with much preliminary work accomplished in section \ref{sec:harmonic frames and orbits}. \\

For a fixed $N$ and $d$, we backwards recursively define the set
\begin{equation*}
\{\alpha_c \in \mathbb{N} \cup \{0\} : c \in \mathbb{N}, c \mid N-1, \text{ and } c \mid d \text{ or } c \mid d-1\}.
\end{equation*}
If $c \mid N-1$, $c \mid d$, and $c > 1$, then
\renewcommand{\theequation}{\thesection.\arabic{equation} $d$}
\begin{equation}\label{eqn:alphac d}
\alpha_c := \frac{(N-1-c)(N-1-2c)\cdots(N-1-(\frac{d}{c}-1)c)}{c^{\frac{d}{c}-1}(d/c)!} - \frac{c}{N-1}\sum_{\substack{c < b < N \\ c|b, \thinspace b|d}} \left(\frac{N-1}{b}\right)\alpha_b,
\end{equation}
where we have used the notation (\ref{eqn:alphac d}) to emphasize
its dependence on the condition $c \mid d$. If $c \mid N-1$, $c \mid d-1$, and $c > 1$, then
\renewcommand{\theequation}{\thesection.\arabic{equation} $d-1$}
\addtocounter{equation}{-1}
\begin{equation}\label{eqn:alphac d-1}
\alpha_c := \frac{(N-1-c)(N-1-2c)\cdots(N-1-(\frac{d-1}{c}-1)c)}{c^{\frac{d-1}{c}-1}((d-1)/c)!} - \frac{c}{N-1}\sum_{\substack{c < b < N \\ c|b, \thinspace b|d-1}} \left(\frac{N-1}{b}\right)\alpha_b.
\end{equation}
Finally, $\alpha_1$ is defined as:
\renewcommand{\theequation}{\thesection.\arabic{equation}}
\begin{equation}\label{eqn:alpha1_frame}
\alpha_1 := \frac{1}{N-1}{N \choose d} - \sum_{\substack{c|d \\ c > 1}} \frac{\alpha_c}{c} - \sum_{\substack{c|d-1 \\ c > 1}} \frac{\alpha_c}{c}.
\end{equation}
\begin{theorem}\label{thm:main2}
Let $N$ be a prime number and let $1 < d < N$. Define the set
\begin{equation*}
\{\alpha_c \in \mathbb{N} \cup \{0\} : c \in \mathbb{N}, c \mid N-1, \text{ and } c \mid d \text{ or } c \mid d-1\},
\end{equation*}
as in equations (\ref{eqn:alphac d}), (\ref{eqn:alphac d-1}), and (\ref{eqn:alpha1_frame}). The total number of harmonic frames for $\mathbb{C}^d$ with $N$ elements is then given by:
\begin{equation}\label{eqn: alpha sum}
\alpha_1 + \sum_{\substack{c|d \\ c > 1}}\alpha_c + \sum_{\substack{c|d-1 \\ c > 1}}\alpha_c.
\end{equation}
\end{theorem}
More concisely, we have the following corollary:
\begin{corollary}
Let $N$ be any prime number and fix $d$ such that $1 < d < N$. Then the number of inequivalent harmonic frames for $\mathbb{C}^d$ with $N$ elements is $O(N^{d-1})$.
\end{corollary}
\begin{proof}
Using equations (\ref{eqn:alphac d}) and (\ref{eqn:alphac d-1}), we
see that $\alpha_c = O(N^s)$, where $c > 1$ and $s \leq \frac{d}{c} - 1 < d - 1$. Therefore, by (\ref{eqn:alpha1_frame}), we see that $\alpha_1 = O(N^{d-1})$, and the corollary follows.
\end{proof}
In the above theorems, the case $d=1$ is omitted, however, it is not hard to see that there are two inequivalent harmonic frames in this case; in fact, there is only one inequivalent harmonic frame for $d=1$ with $N$ distinct vectors.

\section{Harmonic Frames and Orbits}\label{sec:harmonic frames and orbits}

In this section we develop a one-to-one correspondence between
inequivalent harmonic frames and the orbits of a particular set, not
unlike the ideas presented in subsection \ref{subsec:ordered}. First,
however, we come up with an equivalent condition to
(\ref{eqn:equiv2}).\\

We will assume $N$ is prime for the remainder of this paper.

\subsection{A New Equivalence Relation}

When $N$ is prime, every harmonic frame is of the form $U\Phi$, where $U\in\mathcal{U}(\mathbb{C}^d)$ and $\Phi$ is a DFT-FUNTF (see remark \ref{rem:harmonicDFT}). Therefore, finding the number of inequivalent harmonic frames amounts to finding the number of inequivalent DFT-FUNTFs. Toward that end, we simplify (\ref{eqn:equiv2}) to the following:
\begin{theorem}\label{thm:equiv}
If $N$ is prime and $\Phi = \{\phi_0,\ldots ,\phi_{N-1}\}$ and $\Psi = \{\psi_0,\ldots ,\psi_{N-1}\}$ are DFT-FUNTFs, then
\begin{eqnarray}
&& \exists \enspace \sigma_1 \in S_N, \enspace \sigma_2 \in S_d \text{ such that} \nonumber \\
\exists \enspace U\in \mathcal{U}(\mathbb{C}^d) \text{ s.t. } \Phi = U\Psi & \iff & \qquad \phi_m(k) = \psi_{\sigma_1(m)}(\sigma_2(k))\label{eqn:equiv2simple}\\
&& \forall \enspace m = 0,\ldots,N-1, \enspace k = 1,\ldots,d,\nonumber
\end{eqnarray}
where $\phi_m(k)$ denotes the $k^{\text{th}}$ element of the vector $\phi_m$.
\end{theorem}

\begin{proof}
It is clear that if the right hand side of (\ref{eqn:equiv2simple}) holds, then the left hand side must hold as well. Assume then that $\Phi = U\Psi$, and note that
\begin{equation}\label{eqn:equiv2perm1}
\Phi = U\Psi \iff \phi_m = U\psi_{\sigma(m)}, \quad \forall \enspace m
= 0,\ldots,N-1,
\end{equation}
for some permutation $\sigma \in S_N$. Without loss of generality, we may assume that $\sigma(0) = 0$. Indeed, let $\Phi_M$ and $\Psi_M$ be $d \times N$ matrices whose $N$ columns are the vectors $\phi_0,\ldots,\phi_{N-1}$ and $\psi_0,\ldots,\psi_{N-1}$, respectively. Combining (\ref{eqn:equiv2}) and (\ref{eqn:equiv2perm1}), we then have:
\begin{equation}
\Phi \sim_2 \Psi \iff \Phi_M = U\Psi_M P_{\sigma},
\end{equation}
where $P_{\sigma}$ is the $N \times N$ permutation matrix of $\sigma$. By theorem \ref{thm:vw05} there exists a $W \in \mathrm{Sym}(\Psi)$ such that $W\psi_0 = \psi_{\sigma(0)}$. By definition, $W$ is a $d \times d$ matrix that permutes the columns of $\Psi_M$ by acting on the left. Therefore, there exists an $N \times N$ permutation matrix $P_W$ that permutes the columns of $\Psi_M$ in the exact same manner, yet acts on the right. In particular, $W\Psi_M = \Psi_MP_W$, and thus
\begin{equation*}
\Phi_M = UW\Psi_MP_W^{-1}P_{\sigma}.
\end{equation*}
Set $V := UW$ and $P := P_W^{-1}P_{\sigma}$. It is clear that $V$ is a unitary transformation and that $P$ is its associated permutation matrix. Furthermore, $\phi_0 = V\psi_0$, and so we can assume from the start that $\phi_0 = U\psi_0$, i.e., that $\sigma(0) = 0$.\\

Now let $n_1,\ldots,n_d$ denote the column choices of $\Phi$, and consider the following:
\begin{equation}\label{eqn:phi ip}
\langle \phi_m,\phi_0 \rangle = \sum_{k=1}^d e^{2\pi imn_k/N}.
\end{equation}
Letting $l_1,\ldots,l_d$ denote the column choices of $\Psi$, we also have:
\begin{equation}\label{eqn:psi ip}
\langle \phi_m,\phi_0 \rangle = \langle U\psi_{\sigma(m)},U\psi_0 \rangle = \langle \psi_{\sigma(m)},\psi_0 \rangle = \sum_{k=1}^d e^{2\pi i\sigma(m)l_k/N}.
\end{equation}
Define $p_{\phi},p_{\psi} \in \mathbb{Z}[z]/\langle z^N \rangle$ as follows:
\begin{equation}\label{eqn: poly def}
p_{\phi}(z) := \sum_{k=1}^d z^{mn_k} \quad \text{and} \quad p_{\psi}(z) := \sum_{k=1}^d z^{\sigma(m)l_k}.
\end{equation}
By equations (\ref{eqn:phi ip}) and (\ref{eqn:psi ip}), we see that $p_{\phi}(z) = p_{\psi}(z)$ when $z = e^{2\pi i/N}$. In other words, $z = e^{2\pi i/N}$ is a root of the polynomial $p(z) := p_{\phi}(z) - p_{\psi}(z)$. However, since $p \in \mathbb{Z}[z]/\langle z^N\rangle$, and the minimum polynomial of $z = e^{2\pi i/N}$ is $q(z) := \sum_{k=0}^{N-1} z^k$, $p$ must either be an integer multiple of $q$ or the zero polynomial. It is clear, though, that only the latter option is feasible, thus giving
\begin{equation}\label{eqn: poly equal}
p_{\phi}(z) = p_{\psi}(z).
\end{equation}
Combining equations (\ref{eqn: poly def}) and (\ref{eqn: poly equal}), we see there exists a $\sigma_2 \in S_d$ such that
\begin{equation}\label{eqn:perm2}
mn_k = \sigma(m)l_{\sigma_2(k)}, \quad \forall \enspace k = 1,\ldots,d.
\end{equation}
Note that $\sigma_2$ is dependent on the choice of $m$. Taking $m = 1$ in (\ref{eqn:perm2}), one has $n_k = \sigma(1)l_{\sigma_2(k)}$. Letting $\sigma_1(m) := \sigma(1)m$, we have:
\begin{equation}
\phi_m = (e^{2\pi imn_k/N})_{k=1}^d = (e^{2\pi i\sigma_1(m)l_{\sigma_2(k)}})_{k=1}^d = \psi_{\sigma_1(m)}(\sigma_2(k)).
\end{equation}
\end{proof}

\subsection{Inequivalent DFT-FUNTFs and Orbits}

Similar to subsection \ref{subsec:dftfuntfs and orbits}, we now develop a one-to-one correspondence between inequivalent DFT-FUNTFs and the orbits of a particular set. As a matter of notation, we shall denote equivalence classes of (\ref{eqn:equiv2}) by $[\Phi]$, where $\Phi = \{\phi_0,\ldots,\phi_{N-1}\}$ is a DFT-FUNTF representative. By theorem \ref{thm:equiv}, the equivalence classes of (\ref{eqn:equiv2}) are identical to the equivalence classes of the right hand side of (\ref{eqn:equiv2simple}). We now turn our attention to the set with which we will identify the equivalence classes $[\Phi]$.\\

Consider the following equivalence relation on the set $\tilde{\mathbb{Z}}_N^d$,
\begin{equation}\label{eqn:A element equiv}
(n_1,\ldots,n_d) \sim (n_1',\ldots,n_d') \iff \exists \enspace \sigma \in S_d \enspace s.t. \enspace (n_1,\ldots,n_d) = (n_{\sigma (1)}',\ldots,n_{\sigma (d)}').
\end{equation}
Denote an equivalence class of (\ref{eqn:A element equiv}) by the
representative $[n] = [n_1,\ldots,n_d]$, and define $\mathbb{A}_N^d$
as the set of all equivalence classes, i.e.
\begin{equation*}
\mathbb{A}_N^d := \tilde{\mathbb{Z}}_N^d/\sim.
\end{equation*}
It is easy to see $|\mathbb{A}_N^d| = {N \choose d}$. Considering
$\mathbb{Z}_N^{\times}$ as a group and $\mathbb{A}_N^d$ as a set, we
define the group action $\pi_2$,
\begin{equation}\label{eqn:pi2}
\begin{array}{rcl}
\pi_2 :\mathbb{Z}_N^{\times} \times \mathbb{A}_N^d & \rightarrow & \mathbb{A}_N^d\\
(m,[n]) & \mapsto &  m \cdot [n] := [mn_1,\ldots,mn_d].
\end{array}
\end{equation}
The orbits of $\pi_2$ are the sets $\mathcal{O}_{[n]} = \{m \cdot [n] =
[mn_1,\ldots,mn_d] :  m \in
\mathbb{Z}_N^{\times}\}$. The following proposition relates the
equivalence classes of (\ref{eqn:equiv2}) and the orbits of $\pi_2$.
\begin{proposition}\label{prop:1-1equiv2}
There is a one-to-one correspondence between the equivalences classes
of $(\ref{eqn:equiv2})$ and the orbits of $\pi_2$, i.e.
\begin{equation*}
[\Phi_n] \longleftrightarrow \mathcal{O}_{[n]}.
\end{equation*}
\end{proposition}
\begin{proof}
Define the function $F$ as follows:
\begin{equation*}
F([\Phi_n]) = \mathcal{O}_{[n]} = \{[mn_1,\ldots,mn_d] : m \in \mathbb{Z}_N^{\times}\}.
\end{equation*}
We must show that $F$ is well defined, one-to-one, and onto. Surjectivity is clear, so we focus on the first two. To show $F$ is well defined, suppose that $[\Phi_n] = [\Psi_{n'}]$. We want to show $F([\Phi_n]) = F([\Psi_{n'}])$, i.e. $\mathcal{O}_{[n]} = \mathcal{O}_{[n']}$. We have:
\begin{eqnarray*}
[\Phi_n] = [\Psi_{n'}] & \iff & \phi_m(k) = \psi_{\sigma_1(m)}(\sigma_2(k)) \enspace \forall \enspace k = 1,\ldots,d, \enspace \forall \enspace m = 0,\ldots,N-1\\
& \iff & \{\phi_0(k)_{k=1}^d,\ldots,\phi_{N-1}(k)_{k=1}^d\} = \{\psi_0(\sigma_2(k))_{k=1}^d,\ldots,\psi_{N-1}(\sigma_2(k))_{k=1}^d\}\\
& \iff & \{\phi_1(k)_{k=1}^d,\ldots,\phi_{N-1}(k)_{k=1}^d\} = \{\psi_1(\sigma_2(k))_{k=1}^d,\ldots,\psi_{N-1}(\sigma_2(k))_{k=1}^d\}\\
& \iff & \{(mn_1,\ldots,mn_d) : m \in \mathbb{Z}_N^{\times}\} = \{(mn_{\sigma_2(1)}',\ldots,mn_{\sigma_2(d)}') : m \in \mathbb{Z}_N^{\times}\}\\
& \iff & \{[mn_1,\ldots,mn_d] : m \in \mathbb{Z}_N^{\times}\} = \{[mn_1',\ldots,mn_d'] : m \in \mathbb{Z}_N^{\times}\}\\
& \iff & \mathcal{O}_{[n]} = \mathcal{O}_{[n']},
\end{eqnarray*}
where the first equivalence is due to theorem \ref{thm:equiv}, and the third equivalence is because $\phi_0 = \psi_0 = \frac{1}{\sqrt{d}}(1,\ldots,1)$.\\

To prove injectivity, we assume $\mathcal{O}_{[n]} = \mathcal{O}_{[n']}$. According to this assumption, there must exist an $m_0' \in \mathbb{Z}_N^{\times}$ such that $[n_1,\ldots,n_d] = [m_0'n_1',\ldots,m_0'n_d']$. Therefore we have:
\begin{eqnarray*}
\mathcal{O}_{[n]} = \mathcal{O}_{[n']} & \iff & [n_1,\ldots,n_d] = [m_0'n_1',\ldots,m_0'n_d']\\
& \iff & (n_1,\ldots,n_d) = (m_0'n_{\sigma_2(1)}',\ldots,m_0'n_{\sigma_2(d)}')\\
& \iff & (mn_1,\ldots,mn_d) = (mm_0'n_{\sigma_2(1)}',\ldots,mm_0'n_{\sigma_2(d)}'), \enspace \forall \enspace m \in \mathbb{Z}_N^{\times}\\
& \iff & \{(mn_1,\ldots,mn_d) : m \in \mathbb{Z}_N^{\times}\} = \{(mn_{\sigma_2(1)}',\ldots,mn_{\sigma_2(d)}') : m \in \mathbb{Z}_N^{\times}\}\\
& \iff & \{\phi_1(k)_{k=1}^d,\ldots,\phi_{N-1}(k)_{k=1}^d\} = \{\psi_1(\sigma_2(k))_{k=1}^d,\ldots,\psi_{N-1}(\sigma_2(k))_{k=1}^d\}\\
& \iff & \{\phi_0(k)_{k=1}^d,\ldots,\phi_{N-1}(k)_{k=1}^d\} = \{\psi_0(\sigma_2(k))_{k=1}^d,\ldots,\psi_{N-1}(\sigma_2(k))_{k=1}^d\}\\
& \iff & \phi_m(k) = \psi_{\sigma_1(m)}(\sigma_2(k)), \enspace \forall \enspace k = 1,\ldots,d, \enspace m = 0,\ldots,N-1\\
& \iff & [\Phi_n] = [\Psi_{n'}],
\end{eqnarray*}
where the fourth equivalence uses the fact that $\{mm_0' : m \in \mathbb{Z}_N^{\times}\} = \{m : m \in \mathbb{Z}_N^{\times}\}$.
\end{proof}

To conclude this section, we note that when $d = N$, we see $|\mathbb{A}_N^d| = 1$, and so there can be only one orbit. Thus there is only one harmonic frame in this case.

\section{The Number of Orbits of $\mathbb{A}_N^d$}\label{sec:the number of orbits of A}

We begin by counting the number of orbits of $\mathbb{A}_N^d$ under the group action $\pi_2$ for the cases $d=2$ and $d=3$.  We then generalize these results for all $1 < d < N$.

\subsection{Some Examples: d = 2 and d = 3}\label{sec:some examples d=2 and d=3}

\begin{proposition}\label{prop:d=2}
Let $N$ be an odd prime number and let $d = 2$. Then there are $(N+1)/2$ orbits of $\mathbb{A}_N^2$. Therefore, there are $(N+1)/2$ inequivalent harmonic frames for $\mathbb{C}^2$.
\end{proposition}

\begin{proof}
Let $[n] \in \mathbb{A}_N^2$. If $(\mathbb{Z}_N^{\times})_{[n]} = \{1\}$,
then $|\mathcal{O}_{[n]}| = N - 1$. Therefore, if we can find all $[n] \in
\mathbb{A}_N^2$ with non-trivial stabilizer and their corresponding
orbits, we will be able to solve for the total number of
orbits. Assume that $m \cdot [n_1,n_2] = [mn_1,mn_2] = [n_1,n_2]$
for some $m \neq 1$. This implies that
\begin{eqnarray*}
mn_1 & \equiv & n_2 \text{ mod } N,\\
mn_2 & \equiv & n_1 \text{ mod } N.
\end{eqnarray*}
Combining the above equations yields
\begin{displaymath}
\begin{array}{rrcl}
& m^2n_1 & \equiv & n_1 \text{ mod } N\\
\Rightarrow & m & \equiv & \pm 1 \text{ mod } N.
\end{array}
\end{displaymath}

The only valid solution is $m \equiv -1 \text{ mod } N$, which implies $n_2 \equiv -n_1 \text{ mod } N$. Therefore all $[n] \in
\mathbb{A}_N^2$ of the form $[n] = [n_1,-n_1]$, $n_1 \neq 0$, have stabilizer
$\{1,-1\}$. Furthermore, since
\begin{displaymath}
\mathcal{O}_{[1,-1]} = \{m \cdot [1,-1] = [m,-m] : m \in \mathbb{Z}_N^{\times}\},
\end{displaymath}
we see that all such $[n]$ lie in the orbit
$\mathcal{O}_{[1,-1]}$. Finally, these are the only elements of
$\mathbb{A}_N^2$ with nontrivial stabilizer, and thus the number of
orbits of $\mathbb{A}_N^2$ is $x + 1$, where $x$ is the number of orbits of size $N-1$. Therefore,
\begin{eqnarray*}
|\mathbb{A}_N^2| & = & x(N - 1) + |\mathcal{O}_{(1,-1)}|,\\
{N \choose 2} & = & x(N - 1) + (N - 1)/2,\\
N(N - 1)/2 & = & x(N - 1) + (N - 1)/2.
\end{eqnarray*}

Solving for $x$ we get $x = (N - 1)/2$ and so $\mathbb{A}_N^2$ has $x
+ 1 = (N - 1)/2 + 1 = (N + 1)/2$ orbits.  
\end{proof}

\begin{proposition}\label{prop:d=3}
Let $N$ be a prime number, $N > 3$, and let $d = 3$:\\
\begin{enumerate}
\item If $N \equiv 1 \enspace \mathrm{mod} \enspace 3$, then there are $(N^2 - 2N + 7)/6$ orbits of $\mathbb{A}_N^3$.
\item If $N \equiv 2 \enspace \mathrm{mod} \enspace 3$, then there are $(N^2 - 2N + 3)/6$ orbits of $\mathbb{A}_N^3$.
\end{enumerate}
Therefore, if $N \equiv 1 \enspace \mathrm{mod} \enspace 3$, there are $(N^2 - 2N + 7)/6$ inequivalent harmonic frames for $\mathbb{C}^3$, and if $N \equiv 2 \enspace \mathrm{mod} \enspace 3$, there are $(N^2 - 2N + 3)/6$ inequivalent harmonic frames for $\mathbb{C}^3$.
\end{proposition}

\begin{proof}
As in the proof of proposition \ref{prop:d=2}, we are looking for all $[n] \in \mathbb{A}_N^3$ with non-trivial stabilizer and their corresponding orbits. So again we suppose 
\begin{equation}\label{eqn:d=3stab}
m\cdot [n_1,n_2,n_3] = [mn_1,mn_2,mn_3] = [n_1,n_2,n_3],
\end{equation}

for some $m\neq 1$. We now consider two cases:\\

I: Suppose $n_1 = 0$. Then we want $m\cdot [0,n_2,n_3] = [0,mn_2,mn_3] = [0,n_2,n_3]$. But this is just the same situation as the $d=2$ case, and so the elements of $\mathbb{A}_N^3$ of this form with non-trivial stabilizer all lie in the following orbit:
\begin{equation*}
\begin{array}{rcl}
\mathcal{O}_{[0,1,-1]} & = & \{m\cdot [0,1,-1] = [0,m,-m] : m \in \mathbb{Z}_N^{\times} \},\\
|\mathcal{O}_{[0,1,-1]}| & = & (N-1)/2.
\end{array}
\end{equation*}

II: Suppose $n_k \neq 0$ for all $k = 1,2,3$. According to (\ref{eqn:d=3stab}), we have three options for the value of $mn_1$:
\begin{equation*}
mn_1 \equiv \left\{
\begin{array}{l}
n_1 \text{ mod } N, \\
n_2 \text{ mod } N, \\
n_3 \text{ mod } N.
\end{array}
\right.
\end{equation*}
If $mn_1 \equiv n_1 \text{ mod } N$, then $m = 1$, which is trivial and so we disregard this case. Since the order of elements does not matter in $\mathbb{A}_N^3$, there is no difference between $mn_1 \equiv n_2 \text{ mod } N$ and $mn_1 \equiv n_3 \text{ mod } N$, and so we choose the former. Moving on to the value of $mn_2$, we once again have the same three options. However, $mn_2 \equiv n_1 \text{ mod } N$, combined with $mn_1 \equiv n_2 \text{ mod } N$ would imply that $mn_3 \equiv n_3 \text{ mod } N$, thus resulting in $m = 1$. $mn_2 \equiv n_2 \text{ mod } N$ not only would imply $m = 1$, but since $mn_1 \equiv n_2 \text{ mod } N$, would also lead to a contradiction. Therefore $mn_2 \equiv n_3 \text{ mod } N$ must hold, which in turn forces $mn_3 \equiv n_1 \text{ mod } N$. Summarizing, we have
\begin{eqnarray}\label{eqn:d=3array}
mn_1 & \equiv & n_2 \text{ mod } N, \nonumber \\
mn_2 & \equiv & n_3 \text{ mod } N, \\
mn_3 & \equiv & n_1 \text{ mod } N. \nonumber
\end{eqnarray}
Proceeding in a similar fashion to the proof of proposition \ref{prop:d=2}, we see that (\ref{eqn:d=3array}) implies
\begin{equation}\label{eqn:d=3m}
m^3n_1 \equiv n_1 \text{ mod } N.
\end{equation}
We now find all $m \in \mathbb{Z}_N^{\times}$ that satisfy (\ref{eqn:d=3m}). Let $g$ be any primitive root mod $N$, i.e. $\langle g \rangle = \mathbb{Z}_N^{\times}$. Then all nontrivial solutions to (\ref{eqn:d=3m}) are of the form
\begin{equation}
m \equiv g^{(N-1)/3} \text{ mod } N \quad \text{ or } \quad m \equiv g^{2(N-1)/3} \text{ mod } N.
\end{equation}
We have two cases:\\

II.a: If $3$ does not divide $N-1$, i.e. $N \equiv 2 \text{ mod } 3$, then the only solution to (\ref{eqn:d=3m}) is $m = 1$.\\

II.b: If $3$ does divide $N-1$, i.e. $N \equiv 1 \text{ mod } 3$, then the solution set to (\ref{eqn:d=3m}) is:
\begin{equation}
\{1, g^{(N-1)/3}, g^{2(N-1)/3} : g \text{ is a primitive root mod } N \}.
\end{equation}
Therefore all elements in $\mathbb{A}_N^3$ of the form $[n_1,g^{(N-1)/3}n_1,g^{2(N-1)/3}n_1]$, $n_1 \neq 0$, have stabilizer $\{1,g^{(N-1)/3},g^{2(N-1)/3}\}$. Furthermore, all elements of this form lie in the following orbit:
\begin{equation*}
\mathcal{O}_{[1,g^{(N-1)/3},g^{2(N-1)/3}]} = \{[m,mg^{(N-1)/3},mg^{2(N-1)/3}] : m \in \mathbb{Z}_N^{\times} \},
\end{equation*}
where
\begin{equation*}
|\mathcal{O}_{[1,g^{(N-1)/3},g^{2(N-1)/3}]}| = (N-1)/3.
\end{equation*}
Indeed, since we have assumed that $n_1 \neq 0$, there are $N-1$ choices for $n_1$.  However, since the order of elements in the $3$-tuple does not matter, choosing $n_1$ is the same as choosing $g^{(N-1)/3}n_1$ or $g^{2(N-1)/3}n_1$. Therefore there are $(N-1)/3$ elements of this form, and they must all lie in the orbit $\mathcal{O}_{[1,g^{(N-1)/3},g^{2(N-1)/3}]}$. Using the same techniques as in proposition \ref{prop:d=2}, we may now count the number of orbits (recall that $x$ is the number of orbits of size $N-1$):
\begin{enumerate}
\item
If $N \equiv 1 \text{ mod } 3$, then there are $x + 2$ orbits:
\begin{equation*}
|\mathbb{A}_N^3| = x(N-1) + (N-1)/2 + (N-1)/3.
\end{equation*}
Solving for $x$ we get $x + 2 = (N^2 - 2N + 7)/6$.
\item
If $N \equiv 2 \text{ mod } 3$, then there are $x + 1$ orbits:
\begin{equation*}
|\mathbb{A}_N^3| = x(N-1) + (N-1)/2.
\end{equation*}
Solving for $x$ we get $x + 1 = (N^2 - 2N + 3)/6$. 
\end{enumerate}
\end{proof}

\subsection{The Structure of the Orbits of $\mathbb{A}_N^d$}

We now turn our attention to the more general setting, beginning with the following theorem which addresses the order of the orbits of $\mathbb{A}_N^d$ and the form of the elements in the orbits.

\begin{theorem}\label{thm: orbit structure}
Let $N$ be a prime number and let $1 < d < N$. If $\mathcal{O}$ is an orbit of $\mathbb{A}_N^d$ under the group action $\pi_2$, then there exists $c \in \mathbb{N}$ such that $c \mid d$ or $c \mid d-1$, and
\begin{equation}\label{eqn:orbitorder}
|\mathcal{O}| = (N-1)/c.
\end{equation}
Furthermore, let $g$ be a primitive root mod $N$ and set
\begin{equation}\label{eqn:akc}
n_k^c := [n_k,g^{(N-1)/c}n_k,\ldots,g^{(c-1)(N-1)/c}n_k], \quad n_k \neq 0.
\end{equation} 
If $[n] \in \mathcal{O}$, then $[n]$ can be written in the form
\renewcommand{\theequation}{\thesection.\arabic{equation} $c$}
\begin{equation}\label{eqn:aform}
[n] = \left\{
\begin{array}{ll}
[n_1^c,n_2^c,\ldots,n_{d/c}^c] & \text{if } c \mid d,\\
\left.[0,n_1^c,n_2^c,\ldots,n_{(d-1)/c}^c]\right. & \text{if } c \mid d-1.
\end{array}
\right.
\end{equation}
\end{theorem}

\begin{proof}
Let $m \in \mathbb{Z}_N^{\times}$; we determine which elements of $\mathbb{A}_N^d$ are stabilized by $m$ based on the order of $m$. In particular, we will break the argument into two cases: $|m| = c > d$ and $|m| = c \leq d$. We begin with the former.\\

I. Assume $|m| = c > d$.\\

We show that no element in $\mathbb{A}_N^d$ can be stabilized by
$m$. Let $[n] = [n_1,\ldots,n_d] \in \mathbb{A}_N^d$, $n_j \neq 0$ for
all $j = 1,\ldots,d$, and suppose
\begin{eqnarray*}
m \cdot [n] & = & [n],\\
\Longrightarrow m \cdot [n_1,\ldots,n_d] & = & [n_1,\ldots,n_d],\\
\Longrightarrow [mn_1,\ldots,mn_d] & = & [n_1,\ldots,n_d].
\end{eqnarray*}
Therefore, $mn_1 \equiv n_j \text{ mod } N$ for some $j \in \{1,\ldots,d\}$, and because the order of $n_1,\ldots,n_d$ does not matter, without loss of generality we have two choices:
\begin{displaymath}
mn_1 \equiv \left\{
\begin{array}{l}
n_1 \text{ mod } N, \\
n_2 \text{ mod } N.
\end{array}
\right.
\end{displaymath}
If $mn_1 \equiv n_1 \text{ mod } N$, then $m = 1$ and we have a contradiction to the assumption $|m| = c > d$. Therefore, $mn_1 \equiv n_2 \text{ mod } N$ must hold. Continuing, we see that $mn_2 \equiv n_j \text{ mod } N$ for some $j \in \{1,\ldots,d\}$. Without loss of generality, we now have three choices:
\begin{displaymath}
mn_2 \equiv \left\{
\begin{array}{l}
n_1 \text{ mod } N, \\
n_2 \text{ mod } N, \\
n_3 \text{ mod } N.
\end{array}
\right.
\end{displaymath}
If $mn_2 \equiv n_1 \text{ mod } N$, then, combining this with the fact that $mn_1 \equiv n_2 \text{ mod } N$, we see that $m^2 = 1$. However, this contradicts our initial assumption, and so is eliminated from consideration. Similarly, $mn_2 \equiv n_2 \text{ mod } N$ implies $m = 1$ and again leads to a contradiction. Therefore, $mn_2 \equiv n_3 \text{ mod } N$ must hold. Continuing in the same manner, we see:
\begin{displaymath}
\begin{array}{rcccl}
mn_1 & \equiv & mn_1 & \equiv & n_2 \text{ mod } N, \\
mn_2 & \equiv & m^2n_1 & \equiv & n_3 \text{ mod } N, \\
mn_3 & \equiv & m^3n_1 & \equiv & n_4 \text{ mod } N, \\
&& \vdots && \\
mn_{d-1} & \equiv & m^{d-1}n_1 & \equiv & n_d \text{ mod } N.
\end{array}
\end{displaymath}
Therefore, we must have $mn_d \equiv m^dn_1 \equiv n_1 \text{ mod }
N$, which implies $m^d = 1$. Since this contradicts our initial
assumption, we see that no element $m \in \mathbb{Z}_N^{\times}$ with
$|m| = c > d$ can stabilize an element of $\mathbb{A}_N^d$ of the form
$[n_1,\ldots,n_d]$, $n_j \neq 0$ for all $j = 1,\ldots,d$. The
argument for elements of the form $[0,n_1,\ldots,n_{d-1}]$, $n_j \neq
0$ for all $j = 1,\ldots,d-1$, follows similarly.\\

II. Assume $|m| = c \leq d$.\\

We show an element of $\mathbb{A}_N^d$ is stabilized by $m$ if and only if $c \mid d$ or $c \mid d-1$. First, suppose $c \nmid d$  and $c \nmid d-1$. Therefore, there exists $q,r \in \mathbb{Z}$ such that
\begin{equation*}
d = qc + r, \quad q \geq 0, \enspace 1 < r < c.
\end{equation*}
Let $[n] = [n_1,\ldots,n_d] \in \mathbb{A}_N^d$, $n_j \neq 0$ for all
$j = 1,\ldots,d$, and suppose $m \cdot [n] = [n]$. Following the same
argument as in part I of this proof, we see:
\begin{displaymath}
\begin{array}{rcccl}
mn_1 & \equiv & mn_1 & \equiv & n_2 \text{ mod } N, \\
mn_2 & \equiv & m^2n_1 & \equiv & n_3 \text{ mod } N, \\
&& \vdots && \\
mn_{c-1} & \equiv & m^{c-1}n_1 & \equiv & n_c \text{ mod } N, \\
mn_c & \equiv & m^cn_1 & \equiv & n_1 \text{ mod } N,
\end{array}
\end{displaymath}
where the last line results from the fact that $|m| = c \leq d$. Continuing, we see there are two possibilities for $mn_{c+1}$:
\begin{equation*}
mn_{c+1} \equiv \left\{
\begin{array}{l}
n_j \text{ mod } N \text{ for some } j \in \{1,\ldots,c\}, \\
n_{c+2} \text{ mod } N.
\end{array}
\right.
\end{equation*}
If $mn_{c+1} \equiv n_j \text{ mod } N$ for some $j \in \{1,\ldots,c\}$, then $mn_{c+1} \equiv mn_{j-1} \text{ mod } N$, where $n_0 := n_c \text{ mod } N$. However, this would imply that $n_{c+1} \equiv n_{j-1} \text{ mod } N$, a contradiction. Therefore, $mn_{c+1} \equiv mn_{c+2} \text{ mod } N$ must hold, and we can continue with the previous line of reasoning to obtain:
\begin{displaymath}
\begin{array}{rcccl}
mn_{c+1} & \equiv & mn_{c+1} & \equiv & n_{c+2} \text{ mod } N, \\
mn_{c+2} & \equiv & m^2n_{c+1} & \equiv & n_{c+3} \text{ mod } N, \\
&& \vdots && \\
mn_{2c - 1} & \equiv & m^{c-1}n_{c+1} & \equiv & n_{2c} \text{ mod } N, \\
mn_{2c} & \equiv & m^cn_{c+1} & \equiv & n_{c+1} \text{ mod } N.
\end{array}
\end{displaymath}
Continuing with the pattern that has now been established, we arrive at:
\begin{displaymath}
\begin{array}{rcccl}
mn_{qc+1} & \equiv & mn_{qc+1} & \equiv & n_{qc+2} \text{ mod } N, \\
mn_{qc+2} & \equiv & m^2n_{qc+1} & \equiv & n_{qc+3} \text{ mod } N, \\
&& \vdots && \\
mn_{qc+r-1} & \equiv & m^{r-1}n_{qc+1} & \equiv & n_{qc+r} \text{ mod } N.
\end{array}
\end{displaymath}
We must then have:
\begin{equation*}
mn_{qc+r} \equiv m^rn_{qc+1} \equiv n_{qc+1} \text{ mod } N,
\end{equation*}
which in turn implies $m^r = 1$, a contradiction. Therefore, no
element of $\mathbb{A}_N^d$ of the form $[n_1,\ldots,n_d]$, $n_j \neq
0$ for all $j = 1,\ldots,d$, can be stabilized by an $m \in
\mathbb{Z}_N^{\times}$ with $|m| = c \leq d$ such that $c \nmid d$ and
$c \nmid d-1$. The argument for elements of $\mathbb{A}_N^d$ of the
form $[0,n_1,\ldots,n_{d-1}]$, $n_j \neq 0$ for all $j = 1,\ldots,d-1$, follows similarly.\\

We now shift our attention to $m \in \mathbb{Z}_N^{\times}$ such that $c \mid d$ or $c \mid d-1$. In either case there exists a $q \in \mathbb{Z}$ such that,
\begin{equation*}
d = qc, \enspace q \geq 0, \quad \text{or} \quad d-1 = qc, \enspace q \geq 0.
\end{equation*}

Using the same argument that we just completed, we see that if $c \mid
d$ then $m$ stabilizes certain elements of the form
$[n_1,\ldots,n_d]$, $n_j \neq 0$ for all $j = 1,\ldots,d$, whereas if
$c \mid d-1$ then $m$ stabilizes certain elements of the form
$[0,n_1,\ldots,n_{d-1}]$, $n_j \neq 0$ for all $j = 1,\ldots,d-1$. The only difference in reasoning comes at the end, where in this case we do not run into a contradiction. Furthermore, looking back at the above reasoning, we see all elements $[n_1,\ldots,n_d] \in \mathbb{A}_N^d$ stabilized by $m$ must satisfy:
\begin{equation}\label{eqn:aeqns}
mn_{jc+k} \equiv m^kn_{jc+1} \equiv n_{jc+k+1}, \quad \forall \enspace
j = 0,\ldots,q-1, \enspace k = 1,\ldots,c-1,
\end{equation}\\
where $d = qc$ or $d - 1 = qc$, depending on the type of element of $\mathbb{A}_N^d$. By equation (\ref{eqn:aeqns}), any element in $\mathbb{A}_N^d$ stabilized by $m$ can be written in one of two general forms:
\begin{equation}\label{eqn:amform}
[n] = \left\{
\begin{array}{l}
[n_1,mn_1,\ldots,m^{c-1}n_1,\ldots,n_{\frac{d}{c}},mn_{\frac{d}{c}},\ldots,m^{c-1}n_{\frac{d}{c}}] \\

[0,n_1,mn_1,\ldots,m^{c-1}n_1,\ldots,n_{\frac{d-1}{c}},mn_{\frac{d-1}{c}},\ldots,m^{c-1}n_{\frac{d-1}{c}}],
\end{array}
\right.
\end{equation}
where $n_j \neq 0$ and $n_j \neq n_k$ for all $j,k = 1,\ldots,d/c$ or
$j,k = 1,\ldots,(d-1)/c$, depending on the form of $[n]$.
Also, since $|m| = c$, there must exist a primitive root mod $N$, $g$, such that
\begin{equation}\label{eqn:mprimroot}
m = g^{(N-1)/c},
\end{equation}
noting that $c \mid N-1$ since the order of any group element must divide the order of the group. Combining equations (\ref{eqn:amform}) and (\ref{eqn:mprimroot}) gives (\ref{eqn:aform}).\\

In order to prove (\ref{eqn:orbitorder}), we exploit the fact that
\begin{equation}\label{eqn:orbitorder2}
|\mathcal{O}_{[n]}| = \frac{N-1}{|(\mathbb{Z}_N^{\times})_{[n]}|}.
\end{equation}
By (\ref{eqn:orbitorder2}), we need only compute the stabilizer of $[n]$ in $\mathbb{Z}_N^{\times}$, that is $(\mathbb{Z}_N^{\times})_{[n]}$. But (\ref{eqn:amform}) and (\ref{eqn:mprimroot}) easily give
\begin{equation*}
(\mathbb{Z}_N^{\times})_{[n]} = \{g^{l(N-1)/c} : l = 0,\ldots,c-1 \}.
\end{equation*}
Clearly $|(\mathbb{Z}_N^{\times})_{[n]}| = c$, thus proving (\ref{eqn:orbitorder}).
\end{proof}

Before counting the number of orbits $\mathbb{A}_N^d$, we prove two lemmas that simplify this task.  The first shows that the choice of $g$ in (\ref{eqn:akc}) does not matter.
\begin{lemma}\label{lem:g1g2}
If $g_1$ and $g_2$ are two primitive roots mod $N$, and $n_1 \in \mathbb{Z}_N$, $n_1 \neq 0$, then
\begin{equation*} 
[n_1,g_1^{(N-1)/c}n_1,\ldots,g_1^{(c-1)(N-1)/c}n_1] = [n_1,g_2^{(N-1)/c}n_1,\ldots,g_2^{(c-1)(N-1)/c}n_1].
\end{equation*}
\end{lemma}

\begin{proof}
Since $g_1$ and $g_2$ are both primitive roots mod $N$, the sets
\\$\{1,g_1^{(N-1)/c},\ldots,g_1^{(c-1)(N-1)/c}\}$ and
$\{1,g_2^{(N-1)/c},\ldots,g_2^{(c-1)(N-1)/c}\}$ are both complete
solution sets to $x^c \equiv 1 \text{ mod } N$. Therefore
$(n_1,g_2^{(N-1)/c}n_1,\ldots,g_2^{(c-1)(N-1)/c}n_1)$ is a
rearrangement of
$(n_1,g_1^{(N-1)/c}n_1,\ldots,g_1^{(c-1)(N-1)/c}n_1)$, and the lemma
follows.
\end{proof}

The second lemma shows that the representation given by (\ref{eqn:aform}) is not unique and gives the instances where confusion can occur.
\begin{lemma}\label{lem:notunique}
Let $[n] \in \mathbb{A}_N^d$ such that $[n]$ can be written in the form
(4.8 b). If $c \mid b$, then $[n]$ can be written in the form (\ref{eqn:aform}) as well.
\end{lemma}

\begin{proof}
We assume $[n] = [\tilde{n}_1^b,\tilde{n}_2^b,\ldots,\tilde{n}_{d/b}^b]$ and show that we can rewrite this as $[n] = [n_1^c,n_2^c,\ldots,n_{d/c}^c]$. If $[n] = [0,\tilde{n}_1^b,\tilde{n}_2^b,\ldots,\tilde{n}_{(d-1)/b}^b]$ then a similar proof shows how to rewrite this as $[n] = [0,n_1^c,n_2^c,\ldots,n_{(d-1)/c}^c]$.  Recall
\begin{equation*}
\tilde{n}_k^b = [\tilde{n}_k,g^{(N-1)/b}\tilde{n}_k,\ldots,g^{(b-1)(N-1)/b}\tilde{n}_k].
\end{equation*}
Let $a = b/c$ and set $n_1 = \tilde{n}_1$; we want to construct $n_1^c$ out of elements of $\tilde{n}_1^b$, where:
\begin{equation*}
\tilde{n}_1^b = [\tilde{n}_1, g^{(N-1)/b} \tilde{n}_1, \ldots, g^{(b-1)(N-1)/b}\tilde{n}_1^b].
\end{equation*}
Since the order of elements does not matter, we may pick them however we like and rearrange them as we wish. We have that $n_1^c$ is formed out of the following elements of $\tilde{n}_1^b$:
\begin{eqnarray*}
n_1^c & = & [\tilde{n}_1, g^{a(N-1)/b}\tilde{n}_1, \ldots, g^{(c-1)a(N-1)/b}\tilde{n}_1] \\
& = & [n_1, g^{a(N-1)/b}n_1, \ldots, g^{(c-1)a(N-1)/b}n_1] \\
& = & [n_1, g^{a(N-1)/ca}n_1, \ldots, g^{(c-1)a(N-1)/ca}n_1] \\
& = & [n_1, g^{(N-1)/c}n_1, \ldots, g^{(c-1)(N-1)/c}n_1].
\end{eqnarray*}
Likewise, set $n_k = \tilde{n}_k$ for $k = 2,\ldots,d/b$, and construct $n_k^c$ in a similar manner. For the next $c$-tuple, set $n_{\frac{d}{b}+1} = g^{(N-1)/b}\tilde{n}_1$. We then have:
\begin{eqnarray*}
n_{\frac{d}{b}+1}^c & = & [g^{(N-1)/b}\tilde{n}_1, g^{(a+1)(N-1)/b}\tilde{n}_1, \ldots, g^{((c-1)a+1)(N-1)/b}\tilde{n}_1] \\
& = & [n_{\frac{d}{b}+1}, g^{a(N-1)/b}n_{\frac{d}{b}+1}, \ldots, g^{(c-1)a(N-1)/b}n_{\frac{d}{b}+1}] \\
& = & [n_{\frac{d}{b}+1}, g^{(N-1)/c}n_{\frac{d}{b}+1}, \ldots, g^{(c-1)(N-1)/c}n_{\frac{d}{b}+1}].
\end{eqnarray*}
In general,
\begin{equation*}
n_{\frac{jd}{b}+k} = g^{j(N-1)/b}\tilde{n}_k, \quad \forall \enspace j =
0,\ldots,a-1, \enspace k = 1,\ldots,d/b,
\end{equation*}
and the resulting $n_{\frac{jd}{b}+k}^c$ follows similarly.
\end{proof}

\subsection{Proof of Theorem \ref{thm:main2}}

Using theorem \ref{thm: orbit structure} as well as lemmas \ref{lem:g1g2} and \ref{lem:notunique}, we now count the number of orbits of $\mathbb{A}_N^d$. By proposition \ref{prop:1-1equiv2} this is the same as counting the number of inequivalent harmonic frames, and so will complete the proof of theorem \ref{thm:main2}. Let $\gamma_c$ denote the total number of orbits of $\mathbb{A}_N^d$ with $(N-1)/c$ elements. Then, by theorem \ref{thm: orbit structure}, the total number of orbits of $\mathbb{A}_N^d$ is given by
\begin{equation}\label{eqn: gamma sum}
\gamma_1 + \sum_{\substack{c|d \\ c > 1}}\gamma_c + \sum_{\substack{c|d-1 \\ c > 1}}\gamma_c.
\end{equation}
Notice the similarity between equations (\ref{eqn: gamma sum}) and (\ref{eqn: alpha sum}). In fact, we shall prove that
\begin{equation*}
\gamma_c = \alpha_c, \quad \forall \enspace c \in \mathbb{N} \text{ such that } c \mid N-1 \text{ and } c \mid d \text{ or } c \mid d-1.
\end{equation*}
\begin{theorem}\label{thm:numframes}
Let $N$ be a prime number, $1 < d < N$, $c \mid N-1$, $c > 1$, and let $\beta_c$ denote the cumulative order of all orbits of size $(N-1)/c$. Furthermore, let $\gamma_c$ denote the number of orbits of $\mathbb{A}_N^d$ of size $(N-1)/c$, so that
\begin{equation*}
\gamma_c = \frac{c\beta_c}{N-1}.
\end{equation*}
If $c \mid d$, then $\beta_c$ is given by the following backwards recursive formula:
\begin{equation*}
\beta_c = \frac{(N-1)(N-1-c)\cdots(N-1-(\frac{d}{c}-1)c)}{c^{\frac{d}{c}}(d/c)!} - \sum_{\substack{c < b < N \\ c|b, \thinspace b|d}} \beta_b.
\end{equation*}
If $c \mid d-1$, then $\beta_c$ is given by the following backwards recursive formula:
\begin{equation*}
\beta_c = \frac{(N-1)(N-1-c)\cdots(N-1-(\frac{d-1}{c}-1)c)}{c^{\frac{d-1}{c}}((d-1)/c)!} - \sum_{\substack{c < b < N \\ c|b, \thinspace b|d-1}} \beta_b.
\end{equation*}
 The number of orbits of $\mathbb{A}_N^d$ of size $N-1$, denoted $\gamma_1$, is given by:
\begin{equation*}
\gamma_1 = \frac{1}{N-1}{N \choose d} - \sum_{\substack{c|d \\ c > 1}} \frac{\gamma_c}{c} - \sum_{\substack{c|d-1 \\ c > 1}} \frac{\gamma_c}{c}.
\end{equation*}
\end{theorem}

\begin{proof}
We prove the formula for $\beta_c$ when $c \mid d$, noting that the proof is identical for the case when $c \mid d-1$. In order to accomplish this task, we will build up the formula using combinatorial arguments. By theorem \ref{thm: orbit structure}, the elements we are counting are of the form $[n_1^c,n_2^c,\ldots,n_{d/c}^c]$, where
\begin{equation*}
n_k^c = [n_k,g^{(N-1)/c}n_k,\ldots,g^{(c-1)(N-1)/c}n_k], \quad n_k \neq 0.
\end{equation*}
It is clear then, that we have $N-1$ choices for $n_1$, $N-1-c$ choices for $n_2$, $N-1-2c$ choices for $n_3$, and so on. Continuing to the end, we see there are $N-1-(d/c-1)c$ choices for $n_{d/c}$. Furthermore, by lemma \ref{lem:g1g2}, the choice of $g$ does not matter, and so does not add any new elements to count. Therefore, at the moment, we have
\begin{equation*}
(N-1)(N-1-c)\cdots(N-1-(d/c-1)c)
\end{equation*}
elements. Fixing the choice of $n_1$ temporarily, it is clear that if we chose any of $g^{(N-1)/c}n_1,\ldots,g^{(c-1)(N-1)/c}n_1$ instead of $n_1$, then we would have a rearranged version of $n_1^c$. However, the order of elements does not matter in $\mathbb{A}_N^d$, and so these choices are in fact the same as choosing $n_1$. Since there are $c$ such elements (including $n_1$), the number of distinct choices for $n_1$ is in fact $(N-1)/c$. Similarly, we must divide the number of choices for each $n_k$ by a factor of $c$, thus giving
\begin{equation*}
\frac{(N-1)(N-1-c)\cdots(N-1-(d/c-1)c)}{c^{d/c}}
\end{equation*}
elements. Furthermore, again recalling that the order of elements does not matter, we see that the order in which we choose $n_1,\ldots,n_{d/c}$ does not matter either. Consequently, we are now down to
\begin{equation}\label{eqn:number of aform}
\frac{(N-1)(N-1-c)\cdots(N-1-(d/c-1)c)}{c^{d/c}(d/c)!}
\end{equation}
elements. We note that equation (\ref{eqn:number of aform}) gives the number of elements of the form (\ref{eqn:aform}). However, we are not counting all elements of the form (\ref{eqn:aform}), but only those that are in an orbit of size $(N-1)/c$. In fact, by lemma \ref{lem:notunique} any element in an orbit of size $(N-1)/b$, where $c \mid b$ and $b \mid d$, can be rewritten as $[n_1^c,n_2^c,\ldots,n_{d/c}^c]$. Therefore, we must subtract all elements in orbits of size $(N-1)/b$, where $c \mid b$ and $b \mid d$, thus giving:
\begin{equation*}
\beta_c = \frac{(N-1)(N-1-c)\cdots(N-1-(\frac{d}{c}-1)c)}{c^{\frac{d}{c}}(d/c)!} - \sum_{\substack{c < b < N \\ c|b, \thinspace b|d}}\beta_b.
\end{equation*}
The equation for $\gamma_1$ follows from
\begin{equation*}
|\mathbb{A}_N^d| = \gamma_1(N-1) + \sum_{\substack{c|d \\ c > 1}}\left(\frac{N-1}{c}\right)\gamma_c + \sum_{\substack{c|d-1 \\ c > 1}}\left(\frac{N-1}{c}\right)\gamma_c,
\end{equation*}
and the fact that $|\mathbb{A}_N^d| = {N \choose d}$.
\end{proof}

\begin{example}
We apply theorem \ref{thm:numframes} for the case when $d = 3$ and $N \equiv 1 \text{ mod } 3$. In this case, $c = 3$ divides $d$ as well as $N-1$, while $c = 2$ divides $d-1$ as well as $N-1$. Therefore we compute:
\begin{equation*}
\beta_3  =  \frac{N-1}{3} \quad \text{and} \quad \beta_2  =  \frac{N-1}{2},
\end{equation*}
which in turn gives:
\begin{equation*}
\gamma_3 = 1 \quad \text{and} \quad \gamma_2 = 1.
\end{equation*}
Thus,
\begin{eqnarray*}
\gamma_1 & = & \frac{1}{N-1}{N \choose 3} - \frac{1}{3} - \frac{1}{2}\\
& = & \frac{N^2 - 2N}{6} - \frac{2}{6} - \frac{3}{6}\\
& = & \frac{N^2 - 2N - 5}{6},
\end{eqnarray*}
and so the total number of orbits is:
\begin{eqnarray*}
\gamma_1 + \gamma_2 + \gamma_3 & = & \frac{N^2 - 2N - 5}{6} + 1 + 1\\
& = & \frac{N^2 - 2N +7}{6}.
\end{eqnarray*}
Notice this is the same result as proposition \ref{prop:d=3}.
\end{example}

\section{The Symmetry Group}\label{sec:the symmetry group}

We now turn our attention to the symmetry group of prime order harmonic frames. The following theorem proves the existence of a particular subgroup of $\mathrm{Sym}(\Phi_n)$ that is dependent on the generators $n_1, \ldots, n_d$ as well as the order of $\mathcal{O}_{[n]}$.
\begin{theorem}\label{thm:subgroup}
Let $\mathcal{O}_{[n]}$ be an orbit of $\mathbb{A}_N^d$ such that $|\mathcal{O}_{[n]}| = (N-1)/c$, and let $\Phi_n$ be the harmonic frame that corresponds to $\mathcal{O}_{[n]}$ under the one-to-one correspondence described by proposition \ref{prop:1-1equiv2}. Then
\begin{equation*}
\langle\mathrm{diag}(\omega^{n_1},\ldots,\omega^{n_d}), Q \rangle \subseteq \mathrm{Sym}(\Phi_n),
\end{equation*}
where $\mathrm{diag}(\omega^{n_1},\ldots,\omega^{n_d})$ denotes a $d \times d$ matrix with $\omega^{n_1}, \ldots, \omega^{n_d}$ on the diagonal and zeros elsewhere, $\omega = e^{2\pi i/N}$, $Q$ is a $d \times d$ permutation matrix dependent on $\Phi_n$, and $|\langle Q \rangle | = c$.
\end{theorem}
\begin{proof}
Similar to the proof of theorem \ref{thm:equiv}, let $\Phi_M$ denote
the $d \times N$ matrix whose columns are the elements of
$\Phi_n$. We note that $U \in \mathrm{Sym}(\Phi_n)$ if and only if
there exists an $N \times N$ permutation matrix $P$ such that
\begin{equation}\label{eqn:unitary perm}
U\Phi_M = \Phi_M P.
\end{equation}
First using the left hand side of (\ref{eqn:unitary perm}), we have
\begin{equation}\label{eqn:lhs}
(U \Phi_M)^{\star}(U \Phi_M) = \Phi^{\star}_M U^{\star} U \Phi_M =
\Phi^{\star}_M \Phi_M,
\end{equation}
and then equivalently for the right hand side of (\ref{eqn:unitary
  perm}),
\begin{equation}\label{eqn:rhs}
(\Phi_M P)^{\star} (\Phi_M P) = P^{\star} \Phi^{\star}_M \Phi_M P.
\end{equation}
Combining (\ref{eqn:lhs}) and (\ref{eqn:rhs}) we obtain the following
necessary condition for (\ref{eqn:unitary perm}),
\begin{equation*}
\Phi^{\star}_M \Phi_M = P^{\star} \Phi^{\star}_M \Phi_M P,
\end{equation*}
or equivalently,
\begin{equation}\label{eqn:row col perm}
P \Phi^{\star}_M \Phi_M P^{\star}= \Phi^{\star}_M \Phi_M.
\end{equation}
The matrix $\Phi^{\star}_M \Phi_M$ is called the Gram matrix and has the following form:
\begin{equation}
(\Phi^{\star}_M \Phi_M)_{j,k} = \langle \phi_k, \phi_j \rangle =
\sum_{l=1}^de^{2\pi in_l(k-j)/N}, \quad \forall \enspace j,k = 0,\ldots,N-1.
\end{equation}
Two elements $\langle \phi_k, \phi_j \rangle$ and $\langle \phi_{k'},
\phi_{j'} \rangle$ of $\Phi_M^{\star} \Phi_M$ are equal if and only if
\begin{equation}\label{eqn:ip equality}
\sum_{l=1}^de^{2\pi in_l(k-j)/N} = \sum_{l=1}^de^{2\pi in_l(k'-j')/N}.
\end{equation}
Using the same minimum polynomial argument as the one found in the
proof of theorem \ref{thm:equiv}, we see that (\ref{eqn:ip equality})
holds for off diagonal elements of $\Phi_M^{\star} \Phi_M$ if and only if there exists a permutation $\mu \in S_d$ such that
\begin{equation}\label{eqn:ip term by term}
n_l(k-j) \equiv n_{\mu(l)}(k'-j') \text{ mod } N, \quad \forall
\enspace l = 1,\ldots,d, \enspace k \neq j, \enspace k' \neq j'.
\end{equation}
(\ref{eqn:ip term by term}) is in fact the same condition as
(\ref{eqn:A element equiv}), and so we may define the following equivalence relation between  the off diagonal entries of
$\Phi_M^{\star} \Phi_M$ and the elements of $\mathbb{A}_N^d$:
\begin{equation}\label{eqn:off diag equiv relation}
\langle \phi_k, \phi_j \rangle \sim (k-j \text{ mod } N) \cdot [n],
\quad k \neq j.
\end{equation}
For the diagonal entries of $\Phi_M^{\star} \Phi_M$, we define the
representative $[0]$ as
\begin{equation*}
[0] := [\underbrace{0,\ldots,0}_d],
\end{equation*}
and extend our equivalence relation to diagonal elements:
\begin{equation}\label{eqn:diag equiv relation}
\langle \phi_j, \phi_j \rangle \sim [0].
\end{equation}
In order to ease notation, we set $0 \cdot [n] := [0]$, and thus can write $k \cdot [n]$ for all $k \in \mathbb{Z}_N$.
Combining (\ref{eqn:off diag equiv relation}) and (\ref{eqn:diag equiv
  relation}), we see $\sim$ induces an equivalence relation between
the set of inner products, $\{\langle \phi_j, \phi_k \rangle : j,k
= 0,\ldots,N-1\}$, and the set $\mathbb{A}_N^d \cup \{[0]\}$. Defining the matrix $G$ as
\begin{equation}\label{eqn: G formula}
G_{j,k} := (k - j) \cdot [n], \quad \forall \enspace j,k \in \mathbb{Z}_N
\end{equation}
we then have an equivalence relation between $\Phi_M^{\star} \Phi_M$ and $G$:
\begin{equation}\label{eqn:Gram rep}
\Phi_M^{\star} \Phi_M \sim G.
\end{equation}
Combining (\ref{eqn:row col perm}) with (\ref{eqn:Gram rep}) gives the
following necessary condition for (\ref{eqn:unitary perm}) to hold:
\begin{equation}\label{eqn:PG = GP}
PGP^{\star} = G.
\end{equation}
Returning to (\ref{eqn: G formula}), we see $G$ has the form:
\begin{equation}
G = \left(
\begin{array}{cccccc}
a_0 & a_{N-1} & a_{N-2} & \cdots & a_2 & a_1 \\
a_1 & a_0 & a_{N-1} & a_{N-2} & \cdots & a_2 \\
a_2 & a_1 & a_0 & \ddots & \ddots & \vdots \\
\vdots & \ddots & \ddots & \ddots & a_{N-1} & a_{N-2} \\
a_{N-2} & \cdots & a_2 & a_1 & a_0 & a_{N-1} \\
a_{N-1} & a_{N-2} & \cdots & a_2 & a_1 & a_0
\end{array}
\right),
\end{equation}
where $a_k = k \cdot [n]$ for all $k \in \mathbb{Z}_N$. Therefore $G$ is a circulant matrix, and is completely determined by its first column vector. The permutation matrix
\begin{equation}\label{eqn:basic circulant}
C := \left(
\begin{array}{cccccc}
0 & 1 & 0 & 0 & \cdots & 0 \\
0 & 0 & 1 & 0 & \cdots & 0 \\
0 & 0 & 0 & 1 & \cdots & 0 \\
\vdots & \vdots & \vdots & \vdots & \ddots & \vdots \\
0 & 0 & 0 & 0 & \cdots & 1 \\
1 & 0 & 0 & 0 & \cdots & 0
\end{array}
\right).
\end{equation}
is called the basic circulant permutation matrix. A matrix $A$ can be written in the form
\begin{equation}\label{eqn: circ sum}
A = \sum_{k=0}^{N-1}a_k C^k,
\end{equation}
if and only if $A$ is circulant. Therefore, $G$ can be written in the form (\ref{eqn: circ sum}), and as such, it is clear that
\begin{equation*}
C^kG(C^k)^{\star} = G, \quad \forall \enspace k = 0, \ldots, N-1.
\end{equation*}
A simple computation shows that when $U = \mathrm{diag}(\omega^{n_1}, \ldots, \omega^{n_d})$, one has
\begin{equation*}
U^k \Phi_M = \Phi_M C^k, \quad \forall \enspace k = 0, \ldots, N-1.
\end{equation*}
Thus, regardless of the size $\mathcal{O}_{[n]}$,
\begin{equation*}
\mathrm{diag}(\omega^{kn_1}, \ldots, \omega^{kn_d}) \in \mathrm{Sym}(\Phi_n), \quad \forall \enspace k = 0, \ldots, N-1.
\end{equation*}
Note this proves the theorem for the case $|\mathcal{O}_{[n]}| = N-1$. \\

To prove the existence of the matrix $Q \in \mathrm{Sym}(\Phi_n)$ with $|\langle Q \rangle | = c$, suppose that $\Phi_n$ corresponds to $\mathcal{O}_{[n]}$ such that $|\mathcal{O}_{[n]}| = (N-1)/c$, where $c > 1$. Note that by theorem \ref{thm: orbit structure} we have
\begin{equation*}
g^{k(N-1)/c} m \cdot [n] = m \cdot [n], \quad \forall \enspace
m \in \mathbb{Z}_N^{\times}, \enspace k = 1,\ldots, c,
\end{equation*}
and in particular
\begin{equation*}
g^{k(N-1)/c} \cdot [n] = [n], \quad \forall \enspace k = 1, \ldots, c.
\end{equation*}
Therefore, the action of $g^{(N-1)/c}$ on $n$ defines a permutation $\rho \in
S_d$ such that
\begin{equation}\label{eqn:gen perm}
(n_{\rho^k(1)},\ldots,n_{\rho^k(d)}) = g^{k(N-1)/c} \cdot (n_1,\ldots,n_d), \quad \forall \enspace k = 1, \ldots, c.
\end{equation}
Since a permutation of the generators $n_1,\ldots,n_d$ is equivalent to
a permutation of the rows of $\Phi_M$, (\ref{eqn:gen perm}) implies
the existence of a $d \times d$ permutation matrix $Q$, where $Q$ is
the matrix equivalent of $\rho$, as well as an $N \times N$ permutation matrix $P_0$, such that
\begin{equation*}
Q^k \Phi_M = \Phi_M P_0^k, \quad \forall \enspace k = 1,\ldots,c.
\end{equation*}
In other words, $Q \in \mathrm{Sym}(\Phi_n)$, and since
$\mathrm{diag}(\omega^{n_1},\ldots,\omega^{n_d}) \in
\mathrm{Sym}(\Phi_n)$ as well, we must have
\begin{equation*}
 \langle
\mathrm{diag}(\omega^{n_1},\ldots,\omega^{n_d}), Q \rangle \subseteq \mathrm{Sym}(\Phi_n).
\end{equation*}
\end{proof}

\begin{corollary}
Let $\mathcal{O}_{[n]}$ be an orbit of $\mathbb{A}_N^d$ such that $|\mathcal{O}_{[n]}| = N-1$, and let $\Phi_n$ be the harmonic frame that corresponds to $\mathcal{O}_{[n]}$ under the one-to-one correspondence described by proposition \ref{prop:1-1equiv2}. Then
\begin{equation*}
\mathrm{Sym}(\Phi_n) = \langle\mathrm{diag}(\omega^{n_1},\ldots,\omega^{n_d}) \rangle,
\end{equation*}
where $\mathrm{diag}(\omega^{n_1},\ldots,\omega^{n_d})$ denotes a $d \times d$ matrix with $\omega^{n_1}, \ldots, \omega^{n_d}$ on the diagonal and zeros elsewhere, and $\omega = e^{2\pi i/N}$.
\end{corollary}
\begin{proof}
Recall the matrices $G$ and $C$ from the proof of theorem \ref{thm:subgroup}, as given by equations (\ref{eqn: G formula}) and (\ref{eqn:basic circulant}), respectively. We will show that $P = C^k$, $k = 0, \ldots, N-1$, are the only matrices satisfying the necessary condition given by equation (\ref{eqn:PG = GP}). Combining the fact that  $\mathcal{O}_{[n]} = \{m \cdot [n] : m \in \mathbb{Z}_N^{\times}\}$ with the assumption that $|\mathcal{O}_{[n]}| = N - 1$, we have
\begin{equation}\label{eqn: k=k'}
k \cdot [n] = k' \cdot [n] \iff k \equiv k' \text{ mod } N.
\end{equation}
Furthermore, let $\sigma \in S_N$ be the permutation corresponding to the permutation matrix $P$. Equation (\ref{eqn:PG = GP}) can be rewritten as
\begin{equation}\label{eqn: term by term perm}
(\sigma(j) - \sigma(k)) \cdot [n] = (j - k) \cdot [n], \quad \forall \enspace j, k \in \mathbb{Z}_N.
\end{equation}
Combining equations (\ref{eqn: k=k'}) and (\ref{eqn: term by term perm}), one obtains
\begin{equation}\label{eqn: term by term perm 2}
\sigma(j) - \sigma(k) = j - k, \quad \forall \enspace j,k \in \mathbb{Z}_N.
\end{equation}
One can think of (\ref{eqn: term by term perm 2}) as a system of $N^2$ linear equations in the $N$ variables $\sigma(0), \ldots, \sigma(N-1)$, with the two added constraints:
\begin{enumerate}
\item
$\sigma(k) \in \mathbb{Z}_N$ for all $k \in \mathbb{Z}_N$,
\item
$\sigma(j) = \sigma(k)$ if and only if $j = k$.
\end{enumerate}
Clearly (\ref{eqn: term by term perm 2}) is an overdetermined system. However, (\ref{eqn: term by term perm 2}) has $N - 1$ independent equations, given by:
\begin{eqnarray*}
\sigma(1) - \sigma(0) & \equiv & 1 \text{ mod } N \\
\sigma (2) - \sigma(0) & \equiv & 2 \text{ mod } N \\
& \vdots & \\
\sigma(N - 1) - \sigma(0) & \equiv & N - 1 \text{ mod } N.
\end{eqnarray*}
Thus $\sigma(0)$ is a free variable, and can be assigned any value from $\mathbb{Z}_N$. The remaining values of $\sigma$ are then given by:
\begin{equation*}
\sigma(j) \equiv j + \sigma(0) \text{ mod } N, \quad \forall \enspace j = 1, \ldots, N - 1.
\end{equation*}
In conclusion, there are $N$ possible permutations, each corresponding to a different value of $\sigma(0)$. In particular, we have the following correspondence:
\begin{equation*}
\sigma(0) = k \iff P = C^k.
\end{equation*}
\end{proof}
The following conjecture asserts that the subgroup described in theorem \ref{thm:subgroup} in fact is the symmetry group for all prime order harmonic frames, not just those corresponding to orbits of size $N-1$.
\begin{conjecture}
Let $\mathcal{O}_{[n]}$ be an orbit of $\mathbb{A}_N^d$ such that $|\mathcal{O}_{[n]}| = (N-1)/c$, and let $\Phi_n$ be the harmonic frame that corresponds to $\mathcal{O}_{[n]}$ under the one-to-one correspondence described by proposition \ref{prop:1-1equiv2}. Then
\begin{equation*}
\mathrm{Sym}(\Phi_n) = \langle\mathrm{diag}(\omega^{n_1},\ldots,\omega^{n_d}), Q \rangle,
\end{equation*}
where $\mathrm{diag}(\omega^{n_1},\ldots,\omega^{n_d})$ denotes a $d \times d$ matrix with $\omega^{n_1}, \ldots, \omega^{n_d}$ on the diagonal and zeros elsewhere, $\omega = e^{2\pi i/N}$, $Q$ is a $d \times d$ permutation matrix dependent on $\Phi_n$, and $|\langle Q \rangle | = c$.
\end{conjecture}

\section{Closing remarks}\label{sec:closing remarks}

We have enumerated all harmonic frames for $\mathbb{C}^d$ with $N$ elements, where $N$ is a prime number.  A natural question is how to extend these results to all $N$.  Certain problems arise, however, with the techniques used in this paper, since in several instances the fact that $N$ is prime is a key element. In particular, for a general $N$, distinct harmonic frames will arise from groups other than $\mathbb{Z}_N$. Also, even for those harmonic frames that do come from $\mathbb{Z}_N$, new representations must be developed since in general $\mathbb{Z}_N^{\times} \subseteq \{1, \ldots, N\}$.

\section{Acknowledgements}

I would like to thank John Benedetto for introducing me to this problem, as well as Kasso Okoudjou for numerous helpful discussions on this topic.

\bibliography{NumFramesDFTbib}

\nocite{dummit:aa}
\nocite{kovacevic:lbbaf1}
\nocite{kovacevic:lbbaf2}
\nocite{waldron:itf}
\nocite{kovacevic:qfewe}
\nocite{zimmermann:ntffd}

\end{document}